\documentclass{amsart} 

\usepackage[english,french]{babel}
\usepackage[latin1]{inputenc}
\usepackage[T1,OT1]{fontenc}

\usepackage{amsmath}
\usepackage{amsthm}
\usepackage{amssymb}
\usepackage{tikz}

\usepackage{imakeidx} 
\usepackage{appendix}
\usepackage{tikz-cd}
\usepackage{verbatim}
\usepackage{enumitem}
\usepackage{multicol}

\usepackage[backref=page]{hyperref}
\usepackage{cleveref}

\hypersetup{colorlinks=true,linkcolor=blue,anchorcolor=blue,citecolor=blue}

\usepackage{adjustbox}

\usepackage{mathtools}

\newtheorem{thm}{Th\'eor\`eme}[section]
\newtheorem{prop}[thm]{Proposition}
\newtheorem{lemma}[thm]{Lemme}

\newtheorem{question}[thm]{Question}

\theoremstyle{definition}

\theoremstyle{remark}

\newtheorem{rmk}[thm]{Remarque}
\numberwithin{equation}{section}

\newcommand{\F}{\mathbb F}

\newcommand{\G}{\mathbb G}

\renewcommand{\P}{\mathbb P}
\newcommand{\Spec}{\operatorname{Spec}}

\newcommand{\A}{\mathbb A}
\newcommand{\mc}[1]{\mathcal{#1}}

\newcommand{\set}[1]{\left\{#1\right\}}
\renewcommand{\phi}{\varphi}

\newcommand{\on}[1]{\operatorname{#1}}

\title[Pathologies du groupe des classes de R-\'equivalence]{Pathologies du groupe des classes de R-\'equivalence d'un groupe alg\'ebrique lin\'eaire}

\author{Federico Scavia}
\address{Department of Mathematics\\
	University of California\\
	Los Angeles, CA 90095-1555 \\ USA}
\email{scavia@math.ucla.edu}

\makeatletter
\@namedef{subjclassname@2020}{%
	\textup{2020} Mathematics Subject Classification}
\makeatother


\begin{document}

	\maketitle

	\begin{abstract}
		Soient $k_0$ un corps de caract\'eristique $p>0$ et $k=k_0(t)$, o\`u $t$ est transcendant sur $k_0$. On donne un exemple d'un $k$-groupe lisse connexe unipotent $G$ tel qu'il existe une extension s\'eparable finie $F/k$ avec $G(F)/R$ non-commutatif.
	\end{abstract}
	
	\selectlanguage{english}	
	\begin{abstract}
		Let $k_0$ be a field of characteristic $p>0$ and $k=k_0(t)$, where $t$ is transcendental over $k_0$. We give an example of a smooth connected unipotent $k$-group $G$ such that $G(F)/R$ is non-commutative for some finite separable field extension $F$ containing $k$.
	\end{abstract}

	\selectlanguage{french}

	\section{Introduction}
	Soient $k$ un corps et $X$ une $k$-vari\'et\'e, c'est-\`a-dire un $k$-sch\'ema s\'epar\'e de type fini. Deux points $x,y\in X(k)$ sont dits $R$-li\'es s'il existe une $k$-application rationnelle $f:\P^1_k\dashrightarrow X$ d\'efinie en $0$ et $1$ et telle que $f(0)=x$ et $f(1)=y$, et il sont dits $R$-\'equivalents s'il existe des points  $z_1,\dots,z_n\in X(k)$ tels que $z_1=x$, $z_n=y$ et $z_i$ et $z_{i+1}$ sont $R$-li\'es pour tout $i=1,\dots,n-1$. On note par $X(k)/R$ l'ensemble des classes de $R$-\'equivalence; si $K/k$ est une extension de corps, on pose $X(K)/R:=X_K(K)/R$. La notion de $R$-\'equivalence a \'et\'e introduite par Yu. Manin dans \cite{manin1968cubic1} et \cite{manin1969cubic3}; voir aussi le chapitre 1 de \cite{manin1986cubic}. 
	
	Soit $G$ un $k$-groupe, c'est-\`a-dire un $k$-sch\'ema en groupes de type fini, et soit $R(G,e)\subset G(k)$ la classe de $R$-\'equivalence de l'identit\'e $e\in G(k)$. Alors $R(G,e)$ est un sous-groupe normal de $G(k)$ et $G(k)/R=G(k)/R(G,e)$, donc $G(k)/R$ a une structure de groupe compatible avec celle de $G(k)$. La question ouverte suivante est bien connue.
	
	\begin{question}\label{question1}
		Soient $k$ un corps et $G$ un $k$-groupe affine lisse et connexe. Le groupe $G(k)/R$ est-il commutatif?
	\end{question}
	
	Dans cette g\'en\'eralit\'e, la question \ref{question1} a \'et\'e pos\'ee par V. Voskresenski\u{\i} \cite[Complements, \S 1, p. 199]{voskresenskii1977algebraic}. Elle a \'et\'e surtout consid\'er\'ee   pour $G$ r\'eductif.  Une r\'eponse positive est connue dans un certain nombre de cas. 
	
	\begin{itemize}
		\item Si $k$ est infini et $G$ est stablement $k$-rationnel alors le groupe $G(k)/R$ est trivial; voir \cite[Remark p. 180]{voskresenskii2011algebraic}. (L'hypoth\`ese $\on{car}(k)=0$ est inessentielle.)
		\item Supposons $k$ infini et $G$ r\'eductif et quasi-d\'eploy\'e sur $k$. Soit $S$ un $k$-tore d\'eploy\'e maximal de $G$ et $T:=C_G(S)$ le centralisateur de $S$, qui est un $k$-tore maximal de $G$ par \cite[Proposition 20.6]{borel1991linear}. D'apr\`es \cite[\S 18.1]{voskresenskii2011algebraic} ou \cite[Lemme 1.7]{gille2004specialisation} on a $G(k)/R\simeq T(k)/R$, donc $G(k)/R$ est commutatif.
		\item Si $\on{car}(k)=0$ et $G=\on{SL}(A)$, o\`u $A$ est une $k$-alg\`ebre simple centrale, le groupe $G(k)/R$ est commutatif d'apr\`es un th\'eor\`eme de Voskresenski\u{\i} \cite{voskresenskii1977reduced}, qui montre que l'on a $G(k)/R=G(k)/[A^{\times},A^{\times}]=\on{SK}_1(A)$. Le r\'esultat de Voskresenski\u{\i} s'appuie sur l'\'egalit\'e $\on{SK}_1(A)=\on{SK}_1(A\otimes_kk(t))$ \'etablie par V. Platonov \cite{platonov1975tannaka}. Une r\'ef\'erence utile sur ce sujet est \cite[\S 18.2]{voskresenskii2011algebraic}.
		
		\item Dans \cite{merkurjev1996r-equivalence}, A. Merkurjev a d\'etermin\'e explicitement le groupe $G(k)/R$ et montr\'e qu'il est commutatif pour tout corps $k$ de caract\'eristique diff\'erente de $2$ et tout $k$-groupe $G$ adjoint semi-simple classique. 
		
		\item Si $G$ est un $k$-groupe sp\'ecial unitaire, le groupe $G(k)/R$ est connu et commutatif d'apr\`es V. Chernousov et A. Merkurjev \cite{chernousov1998r-equivalence}.
		
		\item Si $k$ est infini et parfait, la commutativit\'e de $G(k)/R$ pour $G=\on{Spin}(q)$, o\`u $q$ est une forme quadratique non-d\'eg\'en\'er\'ee, a \'et\'e d\'emontr\'ee par Chernousov et Merkurjev \cite{chernousov2001r-equivalence}.
	\end{itemize}

	Si $k$ est infini, $X$ est une $k$-vari\'et\'e et $k(u)/k$ est une extension pure de degr\'e de transcendance $1$, alors la fl\`eche naturelle $f:X(k)/R\to X(k(u))/R$ est bijective; voir la remarque \ref{homotopie}. On note $k(\!(u)\!)$ le compl\'et\'e de $k(u)$ par rapport \`a la valuation $u$-adique. Ph. Gille \cite[Corollaire 0.3]{gille2004specialisation} a d\'emontr\'e que, si $\on{car}(k)\neq 2$ et $G$ est un $k$-groupe r\'eductif, alors l'homomorphisme naturel $G(k)/R\to G(k(\!(u)\!))/R$ est bijectif.	
	
	\begin{question}\label{question3}
		Soient $k$ un corps, $G$ un $k$-groupe affine lisse et connexe, $k(u)/k$ 
		une extension pure de degr\'e de transcendance $1$ et $k(\!(u)\!)$ le compl\'et\'e de $k(u)$ par rapport \`a la valuation $u$-adique. L'application naturelle $G(k)/R\to G(k(\!(u)\!))/R$ est-elle un isomorphisme?
	\end{question}	
	
	
	Le but de cette note est de donner des exemples qui montrent que les questions \ref{question1} et \ref{question3} ont une r\'eponse n\'egative dans le cas non-r\'eductif.
	
	\begin{thm}\label{mainthm}
		Soient $p$ un nombre premier, $k_0$ un corps de caract\'eristique $p$ et $k=k_0(t)$, o\`u $t$ est transcendant sur $k_0$. Il existe un $k$-groupe unipotent lisse connexe non-commutatif $G$ de dimension $2$ tel que le groupe $G(k)$ est trivial et $G(K)/R=G(K)$ pour toute extension s\'eparable $K/k$. En particulier:
		\begin{enumerate}
			\item il existe une extension s\'eparable finie $F$ de $k$ telle que le groupe $G(F)/R$ est non-commutatif,
			\item si $k(u)/k$ est une extension pure de degr\'e de transcendance $1$ et $k(\!(u)\!)$ est le compl\'et\'e de $k(u)$ par rapport \`a la valuation $u$-adique, l'homomorphisme de groupes  $G(k)/R\to G(k(\!(u)\!))/R$ n'est pas surjectif.
		\end{enumerate}
	\end{thm}
	Les questions \ref{question1} et \ref{question3} ont une r\'eponse n\'egative pour $G$ comme dans le th\'eor\`eme \ref{mainthm}. Si on pose $k_0=\F_p$ dans le th\'eor\`eme \ref{mainthm}, on obtient des exemples d\'efinis sur des corps globaux.
	
	\medskip
	
	On esquisse la preuve du th\'eor\`eme \ref{mainthm}. Soit $G$ une extension de $k$-groupes
	\[1\to G_1\to G\to G_2\to 1\]
	o\`u $G_1$ et $G_2$ sont unipotents lisses connexes de dimension $1$ et les groupes $G_1(k)$ et $G_2(k)$ sont finis. Comme l'on rappelle dans la section \ref{construction}, l'existence de tels groupes $G_1$ et $G_2$  remonte \`a M. Rosenlicht \cite[p. 46]{rosenlicht1957some} et celle de telles extensions $G$  est due \`a S. Endo \cite[Example 5.10]{kambayashi1974unipotent}	si $p>2$.  
	Une modification de la construction de Endo nous permet d'obtenir des exemples aussi pour $p=2$. Dans tous ces exemples, les groupes finis $G_1(k)$, $G_2(k)$ et $G(k)$ sont m\^eme triviaux. 
	
	Comme $k$ est infini, $G_1$ et $G_2$ ne sont pas $k$-unirationnels et donc, pour toute extension s\'eparable $K/k$, les $K$-groupes  $(G_1)_K$ et $(G_2)_K$ ne sont pas $K$-unirationnels; voir le lemme \ref{achet}. Un argument direct (lemme \ref{no-rat-map}) montre alors que toute application $K$-rationnel\-le $\P^1_K\dashrightarrow G_K$ est constante, et donc $G(K)/R=G(K)$.

	\section*{Remerciements}
	Je remercie Jean-Louis Colliot-Th\'el\`ene pour m'avoir pos\'e les questions \ref{question1} et \ref{question3}, pour ses commentaires utiles sur ce travail et pour les nombreuses conversations math\'ematiques que nous avons eues en 2021. Je remercie Jean-Pierre Serre pour des remarques utiles sur ce texte. Je remercie l'Institut des Hautes \'Etudes Scientifiques pour son hospitalit\'e pendant l'automne 2021.

	\section{Unirationalit\'e et extensions s\'eparables}\label{sec-sep}
	
	Par d\'efinition, une extension de corps $K/k$ (non n\'ecessairement alg\'ebrique) est dite s\'eparable si l'anneau $K\otimes_kk'$ est r\'eduit pour toute extension $k'/k$. L'extension $K/k$ est s\'eparable si est seulement si $K\otimes_kk'$ est r\'eduit pour toute extension finie  de corps $k'/k$. Par exemple, si $k=k_0(t)$, o\`u $t$ est transcendant sur $k_0$, alors l'extension  $k_0(\!(t)\!)/k_0(t)$ est s\'eparable.
	
	Soient $C$ un $k$-groupe commutatif et $K/k$ une extension galoisienne finie. On a une application norme $\nu:R_{K/k}(C_K) \to C$, qui est un homomorphisme de $k$-groupes fid\`element plat. La construction de $\nu$ est sans doute bien connue, mais comme on n'a pas trouv\'e de r\'ef\'erence, on la rappelle ci-dessous.
	
	Soit $\Gamma:=\on{Gal}(K/k)$. Par descente galoisienne, il suffit de construire un homomorphisme $\Gamma$-\'equivariant de $K$-groupes $R_{K/k}(C_K)_K \to C_K$. On a un isomorphisme $\Gamma$-\'equivariant  de $k$-alg\`ebres $K\otimes_kK\simeq K^\Gamma$. Pour toute $K$-alg\`ebre $A$ on a alors une identification fonctorielle $\Gamma$-\'equivariante
	\[R_{K/k}(C_K)_K(A)=C(A\otimes_kK)=C(A\otimes_K K^{\Gamma})=\prod_{\sigma\in \Gamma} C(A\otimes_{K,\sigma}K),\]
	o\`u pour tout $\sigma\in \Gamma$ l'homomorphisme $K\to K$ utilis\'e pour d\'efinir $A\otimes_{K,\sigma}K$ est $\sigma$. Pour tout $\sigma\in \Gamma$, $\sigma^{-1}:K\to K$ est un homomorphisme de $K$-alg\`ebres, o\`u le domaine de $\sigma^{-1}$ est une $K$-alg\`ebre via $\sigma$ et son codomaine est une $K$-alg\`ebre via l'identit\'e de $K$. On tensorise par $A$ et on obtient un homomorphisme de $K$-alg\`ebres $A\otimes_{K,\sigma}K\to A\otimes_{K,\on{id}}K=A$, donc un homomorphisme de groupes $C(A\otimes_{K,\sigma}K)\to C(A)$. Comme $C$ est commutatif, le produit de ces homomorphismes
	\[\prod_{\sigma\in \Gamma}C(A\otimes_{K,\sigma}K)\to C(A)\]
	est encore un homomorphisme de groupes, et on peut v\'erifier qu'il est $\Gamma$-\'equivariant et fonctoriel en $A$. Ceci d\'efinit $\nu$. Par construction, l'homomorphisme $\nu_K$ s'identifie \`a l'homomorphisme de multiplication $(C_K)^{\Gamma}\to C_K$, donc $\nu$ est fid\`element plat.
	
	\begin{lemma}\label{achet}
		Soient $k$ un corps, $C$ un $k$-groupe lisse connexe et commutatif et $K/k$ une extension s\'eparable. Alors $C$ est $k$-unirationnel si et seulement si $C_K$ est $K$-unirationnel.
	\end{lemma}	
	
	\begin{proof}
		Voir \cite[Theorem 2.3]{achet2019unirational}. On donne ici une d\'emonstration qui n'utilise pas le groupe de Picard rigidifi\'e. Il suffit de montrer que si $C_K$ est $K$-unirationnel, alors $C$ est $k$-unirationnel.
		
		Supposons d'abord que $K/k$ est alg\'ebrique et galoisienne.  L'argument qui suit nous a \'et\'e sugg\'er\'e par Jean-Louis Colliot-Th\'el\`ene. On se r\'eduit imm\'ediatement au cas o\`u $K/k$ est galoisienne finie de degr\'e $d\geq 1$. Il existe un ouvert dense $U\subset \A^n_K$ et un $K$-morphisme dominant
		$f:U \to C_K$. Alors $R_{K/k}(U)$ est un ouvert dense de $R_{K/k}(\A^n_K)\cong \A^{dn}_k$ et, comme $K/k$ est s\'eparable, le $k$-morphisme $R_{K/k}(f)$ est dominant. Comme le $k$-groupe $C$ est commutatif et l'extension $K/k$ est galoisienne finie, l'application norme $\nu:R_{K/k}(C_K) \to C$ est d\'efinie et fid\`element plate. On conclut que le morphisme compos\'e $\nu\circ R_{K/k}(f)$ est dominant, et donc que $C$ est $k$-unirationnel.
		
		On d\'eduira le cas g\'en\'eral par un argument d'\'etalement standard; cf. \cite[Lemma 2.2]{achet2019unirational}. Soit $K/k$ une extension s\'eparable arbitraire. D'apr\`es le cas pr\'ec\'edent, on peut remplacer $k$ et $K$ par leurs cl\^{o}tures s\'eparables et donc supposer que $k$ est s\'eparablement clos. Comme $K/k$ est limite directe de sous-extensions $K'/k$ de type fini, on peut supposer $K/k$ de type fini. Si $U\subset \A^n_K$ est un ouvert dense et $\phi:U\to C_K$ est un morphisme dominant, on peut trouver une $k$-alg\`ebre lisse $A\subset K$, un ouvert dense $\mc{U}\subset \A^n_A$ tel que $\mc{U}_K=U$ et un morphisme dominant 
		de $A$-sch\'emas $\psi:\mc{U}\to C_A$ tel que $\psi_K=\phi$. Soit $\Sigma\subset \Spec(A)$ l'ensemble des points $P\in \Spec(A)$ tels que $\psi_P:\mc{U}_P\to C_{k(P)}$ est dominante. Alors $\Sigma$ est constructible \cite[Proposition 9.6.1(ii)]{ega4} et contient le point g\'en\'erique de $\Spec(A)$, donc il contient un ouvert dense $V\subset \Spec(A)$. Comme $A$ est lisse sur $k$ et $k$ est s\'eparablement clos, $V(k)\neq\emptyset$. Si $P\in V(k)$, $\mc{U}_P$ est un ouvert dense de $\A^n_k$ et $\psi_P:\mc{U}_P\to C$ est le morphisme dominant cherch\'e.
	\end{proof}	
	
	\begin{lemma}\label{no-rat-map}
		Soit $k$ un corps. Supposons donn\'ee une suite exacte courte de $k$-groupes \[1\to G_1\to  G\xrightarrow{\pi} G_2\to 1,\] o\`u les $k$-groupes $G_1$ et $G_2$ sont  lisses affines connexes de dimension $1$ et ne sont pas $k$-unirationnels. Si $K/k$ est une extension s\'eparable, alors toute $K$-application rationnelle $\P^1_K\dashrightarrow G_K$ est constante. En particulier $G(K)/R=G(K)$.
	\end{lemma}
	
	\begin{proof}
		Soit $\phi:\P^1_K\dashrightarrow G_K$ une 
		$K$-application rationnelle; on veut montrer que $\phi$ est un morphisme constant. Comme $K$ est infini, il existe $g\in G(K)$ dans l'image de $\phi$. En composant $\phi$ avec la multiplication par $g^{-1}$ on peut alors supposer que l'\'el\'ement neutre $1_G\in G(K)$ est dans l'image de $\phi$. 
		
		Comme $G_1$ et $G_2$ sont connexes lisses et de dimension $1$, ils sont commutatifs (on peut le v\'erifier sur la cl\^oture alg\'ebrique de $k$). D'apr\`es le lemme \ref{achet}, $(G_1)_K$ et $(G_2)_K$ ne sont pas $K$-unirationnels. On d\'eduit que $\pi_K\circ \phi$ est constant et donc, comme $1_G$ est dans l'image de $\phi$, que $\pi_K\circ \phi=1_{G_2}$, c'est-\`a-dire que $\phi$ factorise par $(G_1)_K$. Comme $(G_1)_K$ n'est pas $K$-unirationnel, on conclut que $\phi$ est constant.
	\end{proof}

	\section{L'exemple}\label{construction}
	
	Le but de cette section est de prouver la proposition suivante.
	
	\begin{prop}\label{exemples}
		Soient $k_0$ un corps de caract\'eristique $p>0$ et $k=k_0(t)$, o\`u $t$ est transcendant sur $k_0$. Il existe une extension centrale de $k$-groupes
		\begin{equation}\label{suite}
			1\to G_1\to G\to G_2\to 1,
		\end{equation}
		o\`u les $k$-groupes $G_1$ et $G_2$ sont lisses connexes unipotents de dimension $1$, $G_1(k)$ et $G_2(k)$ sont triviaux, et le $k$-groupe $G$ est non-commutatif.
	\end{prop}
	
	Soient $G_1$ et $G_2$ deux groupes commutatifs sur un corps $k$. On \'ecrit  les op\'erations de groupe de $G_1$ et $G_2$ additivement. Soit $h:G_2\times G_2\to G_1$ un morphisme bi-additif, c'est-\`a-dire
	\[h(g_2+g_2',g_2'')=h(g_2,g_2'')+h(g_2',g_2''),\qquad h(g_2,g_2'+g_2'')=h(g_2,g_2')+h(g_2,g_2'')\]
	pour toute $k$-alg\`ebre $A$ et tous $g_2,g_2',g_2''\in G_2(A)$.
	Alors 
	\begin{equation}\label{zero-g2}
		h(0,g_2)=h(g_2,0)=0	
	\end{equation} pour toute $k$-alg\`ebre $A$ et tout $g_2\in G_2(A)$, et $h$ satisfait la condition de $2$-cocycle  (o\`u $G_1$ a $G_2$-action triviale)
	\begin{equation}\label{2-cocycle} h(g_2+g_2',g_2'')+h(g_2,g_2')=h(g_2,g_2'+g_2'')+h(g_2',g_2'')
	\end{equation}
	pour toute $k$-alg\`ebre $A$ et tous $g_2,g_2',g_2''\in G_2(A)$.
	On peut donc construire une extension centrale  de $k$-groupes (\ref{suite}): on pose $G:=G_1\times G_2$ comme $k$-sch\'ema, et on d\'efinit son op\'eration de groupe par
	\[(g_1,g_2)\cdot (g_1',g_2'):=(g_1+g_1'+h(g_2,g_2'),g_2+g_2').\]
	L'identit\'e de $G$ est $(0,0)$ et 
	\[(g_1,g_2)^{-1}=(-g_1-h(g_2,-g_2),-g_2).\]
	L'associativit\'e de cette operation est assur\'ee par (\ref{2-cocycle}), et le fait que $(0,0)$ est l'identit\'e suit de (\ref{zero-g2}). Le $k$-groupe $G$ est commutatif si et seulement si $h$ est sym\'etrique, c'est-\`a-dire $h(g_2,g_2')=h(g_2',g_2)$ pour toute $k$-alg\`ebre $A$ et tous $g_2,g_2'\in G_2(A)$.
	
	Pour montrer la proposition \ref{exemples}, il suffit alors de construire $G_1$ et $G_2$ lisses connexes unipotents de dimension $1$ sur $k$ avec $G_1(k)$ et $G_2(k)$ triviaux et un morphisme bi-additif $h:G_2\times G_2\to G_1$ qui n'est pas sym\'etrique.

	\medskip
	
	On commence par la construction de $G_1$ et $G_2$. Le premier exemple de groupe affine lisse et connexe $G$ sur un corps infini $k$ avec $G(k)$ fini est du \`a Rosenlicht \cite[p. 46]{rosenlicht1957some}. Le lemme suivant est une l\'eg\`ere g\'en\'eralisation de l'exemple de Rosenlicht, et sa preuve suit celle de J. Tits dans \cite[p. 207]{tits2013resumes}. On donnera d'autres exemples dans la remarque \ref{gabber}.
	
	\begin{lemma}\label{finite-group}
		Soient $k_0$ un corps de caract\'eristique $p>0$, $k:=k_0(t)$, $a,b\in k_0^{\times}$ et $m,n\geq 1$ des entiers tels que $p^m>2$. Consid\'erons le $k$-groupe $\G_{\on{a},k}^2$ avec cordonn\'ees $X$ et $Y$ et le $k$-groupe connexe lisse unipotent
		\[U:=\set{X=atX^{p^m}+bY^{p^n}}\subset \G_{\on{a},k}^2.\]
		Alors $U(k)=\set{(0,0)}$.
	\end{lemma}
	
	\begin{proof}
		Soient $x,y\in k$ tels que $x=atx^{p^m}+by^{p^n}$. On d\'erive 
		suivant la variable	$t$. Comme $\frac{d(f^p)}{dt}=0$ pour tout $f\in k$, cela nous donne  $\frac{dx}{dt}=ax^{p^m}$. Soient $u,v\in k_0[t]$ premiers entre eux tels que $x=\frac{u}{v}$. Alors $\frac{du}{dt}v-u\frac{dv}{dt}=\frac{u^{p^m}}{v^{p^m-2}}$. Comme $p^m>2$, on obtient $v\in k_0^{\times}$ et donc $x\in k_0[t]$. Si $x$ est non nul de degr\'e $d\geq 1$, l'\'egalit\'e $\frac{dx}{dt}=ax^{p^m}$ entra\^{i}ne $d-1=p^md$, ce qui est impossible, donc $x$ est constant, donc nul. Ceci montre que la solution $(0,0)$ est la seule.
	\end{proof}

	
	La construction suivante est due \`a Endo; voir \cite[Example 5.10]{kambayashi1974unipotent}. Historiquement, il s'agit du premier exemple d'un $k$-groupe unipotent ploy\'e non-commutatif de dimension $2$. Soient $k_0$ un corps de caract\'eristique $p>0$, $k:=k_0(t)$ et $1\leq m\leq n$ 
	des entiers. Consid\'erons le $k$-groupe $\G_{\on{a},k}^2$ avec cordonn\'ees $X$ et $Y$ et les $k$-groupes unipotents commutatifs connexes lisses $1$-dimensionnels
	\[G_1:=\set{X=tX^{p^m}+Y^{p^m}}\subset \G^2_{\on{a},k},\qquad G_2:=\set{X=-tX^{p^{2m}}+Y^{p^n}}\subset \G_{\on{a},k}^2.\]
	Posons
	\[h:G_2\times G_2\to G_1,\qquad h((x,y),(x',y')):=(x^{p^m}x'-x(x')^{p^m},x(y')^{p^{n-m}}-y^{p^{n-m}}x').\]
	A priori le codomaine de $h$ devrait \^{e}tre $\G_{\on{a},k}^2$, mais un calcul facile montre que l'image de $h$ est effectivement contenue dans $G_1$. On peut v\'erifier que $h$ est un morphisme bi-additif. Il n'est pas sym\'etrique si $p>2$. Comme $m,n\geq 1$, d'apr\`es le lemme \ref{finite-group} les groupes $G_1(k)$ et $G_2(k)$ sont triviaux. Ceci d\'emontre la proposition \ref{exemples} pour tout $p>2$.
	
	
	\medskip
	
	Pour conclure la d\'emonstration de la proposition \ref{exemples}, on doit traiter le cas $p=2$. On garde les notations pr\'ec\'edentes, et on suppose $p=2$ et $m$ impair. On commence en supposant $k_0=\F_2$, donc $k=\F_2(t)$. On \'ecrit $k':=\F_4(t)$ et on choisit $\zeta\in \F_4-\F_2$.  Consid\'erons le morphisme 
	\[f:G_2\to G_1,\qquad f(x,y):= (x^{2^m+1},xy^{2^{n-m}}).\]
	Un calcul direct montre
	\begin{equation}\label{h=df}
		h(g_2,g_2')=f(g_2+g_2')+f(g_2)+f(g_2').
	\end{equation}
	Comme $m$ est impair et $\zeta$ est une racine cubique primitive de $1$, $\zeta^{2^m}=\zeta^{-1}\neq\zeta$. On d\'eduit par des calculs \'el\'ementaires que le morphisme
	\[h_{\zeta}:(G_2)_{k'}\times (G_2)_{k'}\to (G_1)_{k'},\qquad h_{\zeta}(g_2,g'_2):= h(g_2,\zeta g'_2)\]
	est un morphisme bi-additif qui n'est pas sym\'etrique, et donc qui d\'efinit une extension non-commutative
	\begin{equation}\label{f4t}1\to (G_1)_{k'}\to G_{\zeta}\to (G_2)_{k'}\to 1\end{equation}
	sur $k'$. On va  montrer que (\ref{f4t}) est d\'efinie sur $k$ par un argument de descente inspir\'e par O. Gabber \cite[Example 2.10]{conrad2015structure}. Soit $\sigma$ l'\'el\'ement non-trivial de $\on{Gal}(\F_4/\F_2)$. Comme $\sigma(\zeta)=\zeta^{-1}=\zeta+1$, on a une identification $\sigma^*(G_{\zeta})=G_{\zeta+1}$. D'autre part, en utilisant le fait que $\zeta^{-1}=\zeta+1$, on peut v\'erifier que \[h_{\zeta+1}=h_{\zeta}+h_{k'},\] et un calcul direct utilisant (\ref{h=df}) montre alors que l'on a un isomorphisme d'extensions
	\[G_{\zeta}\xrightarrow{\sim} G_{\zeta+1},\qquad (g_1,g_2)\mapsto (g_1+f(g_2),g_2).\]
	On a donc construit un isomorphisme $\phi:G_{\zeta}\to \sigma^*(G_{\zeta})$. On v\'erifie que le compos\'e $\sigma^*(\phi)\circ\phi:G_{\zeta}\to G_{\zeta}$ est l'identit\'e, c'est-\`a-dire $\phi$ d\'efinit une donn\'ee de descente relative \`a l'extension $k'/k$. On obtient une extension $G$ de $G_2$ par $G_1$ d\'efinie sur $k$ et qui donne (\ref{f4t}) par changement de base \`a $k'$. D'apr\`es le lemme \ref{finite-group}, $G_1(k)$ et $G_2(k)$ sont triviaux. Ceci d\'emontre la proposition \ref{exemples} si $k_0=\F_2$. Si $k_0$ est arbitraire, encore par le lemme \ref{finite-group} on obtient l'exemple cherch\'e sur $k=k_0(t)$ \`a partir de celui sur $\F_2(t)$ par changement de base de $\F_2(t)$ \`a $k$. On a donc d\'emontr\'e la proposition \ref{exemples}.

	\begin{rmk}\label{gabber}
		La construction suivante, due \`a Gabber, a d\'ej\`a paru dans la litt\'era\-ture; voir \cite[Example 2.10]{conrad2015structure} et \cite[\S 2]{rosengarten2021pathological}. Soit $k=k_0(t)$, o\`u $t$ est transcendant sur $k_0$. Consid\'erons le $k$-groupe $\G_{\on{a},k}^2$ avec  coordonn\'ees
		$X$ et $Y$ et les $k$-groupes unipotents commutatifs connexes lisses $1$-dimensionnels
		\[G_1:=\set{X=-X^p-tY^p}\subset \G_{\on{a},k}^2,\qquad G_2:=\set{X=X^{p^2}+tY^{p^2}}\subset \G_{\on{a},k}^2.\]
		On veut montrer que $G_1(k)$ et $G_2(k)$ sont finis pour tout $p>2$. Plus en g\'en\'eral, on montrera que pour tout  $k$-groupe	
		\[V:=\set{X=aX^{p^m}+btY^{p^n}}\subset \G_{\on{a},k}^2,\] 
		o\`u $m$ est un entier tel que $p^m>2$, $n\geq 1$ et $a,b\in k_0^{\times}$, le groupe $V(k)$ est fini. Soient $x,y\in k$ tels que $x=ax^{p^m}+bty^{p^n}$. On \'ecrit cela comme $(x-ax^{p^m})\frac{1}{t}=by^{p^n}$ et on d\'erive 
		suivant la variable $t$. Cela nous donne $\frac{dx}{dt}t=x-ax^{p^m}$. Soient $u,v\in k_0[t]$ premiers entre eux tels que $x=\frac{u}{v}$. Alors $(\frac{du}{dt}v-u\frac{dv}{dt})t=uv+a\frac{u^{p^m}}{v^{p^m-2}}$. Comme $p^m>2$, on obtient $v\in k_0^{\times}$ et donc $x\in k_0[t]$. Si $x$ est de degr\'e $d\geq 1$, l'\'egalit\'e $\frac{dx}{dt}t=x-ax^{p^m}$ entra\^{i}ne $d=p^md$, ce qui est impossible, donc $x\in k_0$. Comme $ty^{p^n}=x-ax^{p^m}\in k_0$, on d\'eduit $y=0$ et $x-ax^{p^m}=0$. Ceci montre que $V(k)$ est fini; en particulier $G_1(k)$ et $G_2(k)$ sont finis.
		
		Soit
		\[h:G_2\times G_2\to G_1,\qquad h((x,y),(x',y')):=(x(x')^p-x^px',x(y')^p-x'y^p)\]
		Des calculs faciles montrent que l'image de $h$ est effectivement contenue dans $G_1$ et que $h$ est un morphisme bi-additif. Il est non-sym\'etrique si $p>2$. Le $k$-groupe $G$ associ\'ee \`a $h$ comme ci-dessus est donc non-commutatif si $p>2$. Ceci donne une d\'emonstration alternative de la proposition \ref{construction} (avec $G_1(k)$ et $G_2(k)$ non n\'ecessairement triviaux mais finis) pour $p>2$.
		
		Si $p=2$, l'exemple ci-dessus est commutatif. Par un argument de descente on peut obtenir des extensions $G$ de $G_2$ par $G_1$ telles que $G$ est non-commutatif; voir encore \cite[Example 2.10]{conrad2015structure}. D'autre part, $G_1\subset \A^2_{k}$ est une conique avec un point $k$-rationnel, donc birationnelle \`a $\P^1_{k}$. On d\'eduit que $G_1(k)/R=1$.
		Pour toute extension centrale $G$ de $G_2$ par $G_1$ obtenue selon la proc\'edure d\'ecrite au d\'ebut de cette section, on a $G=G_1\times G_2$ comme $k$-vari\'et\'es et alors \[G(K)/R=G_1(K)/R\times G_2(K)/R=G_2(K)/R\] pour toute extension $K/k$.  Le groupe $G(K)/R$ est donc commutatif pour toute extension $K/k$, c'est-\`a-dire on ne peut pas utiliser $G$ pour donner une r\'eponse n\'egative \`a la question \ref{question1} dans le cas $p=2$.
	\end{rmk}

	\section{D\'emonstration du th\'eor\`eme \ref{mainthm}}

	On d\'emontre le th\'eor\`eme \ref{mainthm}. Soient $G_1$, $G_2$ et $G$ comme dans la proposition \ref{exemples}. En particulier, les groupes $G_1(k)$, $G_2(k)$ et $G(k)$ sont triviaux. Comme $k$ est infini et $G_1$ et $G_2$ ont au plus un nombre fini de points $k$-rationnels, ils ne sont pas $k$-unirationnels. Le lemme \ref{no-rat-map} donne alors $G(K)/R=G(K)$ pour toute extension s\'eparable $K/k$. 
	
	Soit $k_{\on{sep}}$ la cl\^{o}ture s\'eparable de $k$. Comme $G$ est lisse sur $k$, $G(k_{\on{sep}})\subset G_{k_{\on{sep}}}$ est dense. Comme $G$ est non-commutatif, le groupe $G(k_{\on{sep}})$ n'est pas commutatif. Le groupe $G(k_{\on{sep}})$ est la r\'eunion des $G(F)$, o\`u $F/k$ est finie et s\'eparable. On peut alors trouver $F/k$ finie et s\'eparable telle que $G(F)$ n'est pas commutatif. Ceci d\'emontre (1).
	
	Comme l'extension de corps $k(\!(u)\!)/k$ est s\'eparable, l'homomorphisme $G(k)/R\to G(k(\!(u)\!))/R$ s'identifie \`a $G(k)\to G(k(\!(u)\!))$. Pour montrer (2) il suffit alors de prouver que le groupe $G(k(\!(u)\!))$ est non-trivial. Ce groupe est m\^eme dense dans $G_{k(\!(u)\!)}$: ceci suit du fait que $k(\!(u)\!)$ est un corps fertile (``large field'' en anglais), $G$ est irr\'eductible et lisse et $G(k(\!(u)\!))\neq\emptyset$. Ceci prouve (2) et conclut la d\'emonstration du th\'eor\`eme \ref{mainthm}.

	\begin{rmk}\label{homotopie}
		Soient $k$ un corps infini, $X$ une $k$-vari\'et\'e et $k(u)/k$ une extension pure de degr\'e de transcendance $1$. Comme on n'a pas trouv\'e de r\'ef\'erence dans la litt\'erature, on donne ici la preuve du fait bien connu suivant, qui avec le r\'esultat de Gille \cite[Corollaire 0.3]{gille2004specialisation} motive la question \ref{question3}: {\em l'application naturelle $X(k)/R\to X(k(u))/R$ est bijective.}
		
		On commence par montrer que cette fl\`eche est injective. Notons $v$ la cordonn\'ee sur $\P^1_{k(u)}$ et $k(u,v)/k(u)$ le corps de fonctions de $\P^1_{k(u)}$.  Si $x,y\in X(k)$ sont $R$-\'equivalents dans $X(k(u))$, il existe $z_1(u),\dots,z_n(u)\in X(k(u))$ tels que $z_1(u)=x$ et $z_n(u)=y$, et des $k(u)$-applications rationnelles $\phi_i=\phi_i(u,v):\P^1_{k(u)}\dashrightarrow X_{k(u)}$ d\'efinies en $v=0,1$ telles que $\phi_i(u,0)=z_i(u)$ et $\phi_i(u,1)=z_{i+1}(u)$ pour $i=1,\dots,n-1$. Comme $k$ est infini, il existe $u_0\in k$ tel que les $z_i(u)$ et les $\phi_i(u,v)$ se sp\'ecialisent en $u=u_0$. Par sp\'ecialisation en $u=u_0$ on obtient une $R$-\'equivalence entre $x$ et $y$ dans $X(k)$. Ceci montre l'injectivit\'e de $X(k)/R\to X(k(u))/R$. 
		
		On passe \`a  la surjectivit\'e. Soit $z=z(u):\A^1_k\dashrightarrow X$ un $k(u)$-point de $X$. Soit $Z\subset \A^1_k$ le compl\'ementaire de l'ouvert de d\'efinition de $z(u)$; on peut supposer que $0$ n'appartient pas \`a $Z$, c'est-\`a-dire $z(u)$ est d\'efinie en $u=0$. Si on note par $u$ et $v$ les cordonn\'ees de $\A^2_k$, on obtient une $k$-application rationnelle $\phi:\A^2_k\dashrightarrow X$ en posant $\phi(u,v):=z(uv)$, c'est-\`a-dire $\phi:=z\circ\psi$, o\`u $\psi:\A^2_k\to \A^1_k$ est donn\'e par $\psi(u,v)=uv$. Alors $\phi$ est d\'efinie en d\'ehors du ferm\'e $\psi^{-1}(Z)$, donc son lieu de d\'efinition a intersection non-vide avec la droite $v=0$ et la droite $v=1$. On a $\phi(u,1)=z(u)$ et $\phi(u,0)=z(0)$ dans $X(k(u))$. Si on pose $\phi':=\phi_{k(u)}:\P^1_{k(u)}\dashrightarrow X_{k(u)}$, ceci montre que $z(u)$ et $z(0)$ sont $R$-li\'es par $\phi'$. Comme $z(0)$ est dans l'image de $X(k)\to X(k(u))$, on d\'eduit que $z(u)$ est dans l'image de $X(k)/R\to X(k(u))/R$. On conclut que $X(k)/R\to X(k(u))/R$ est surjective.
	\end{rmk}


\begin{thebibliography}{KMT74}
		
		\bibitem[Ach19]{achet2019unirational}
		R. Achet.
		\newblock {Unirational algebraic groups}.
		\newblock https\string://hal.archives-ouvertes.fr/hal-02358528, novembre 2019.
		
		\bibitem[Bor91]{borel1991linear}
		A. Borel.
		\newblock {\em Linear algebraic groups}.
		\newblock Graduate Texts in Mathematics, Springer, New York, NY, deuxi\`eme edition, 1991.
		
		
		\bibitem[CM98]{chernousov1998r-equivalence}
		V. Chernousov et A. Merkurjev.
		\newblock {$R$}-equivalence and special unitary groups. 
		\newblock {\em J. Algebra} 209 (1998), no. 1, 175--198.
		
		
		\bibitem[CM01]{chernousov2001r-equivalence}
		V. Chernousov et A. Merkurjev.
		\newblock {$R$}-equivalence in spinor groups.
		\newblock {\em J. Amer. Math. Soc.}, 14(3)\string:509--534, 2001.
		
		
		
		
		\bibitem[Con15]{conrad2015structure}
		B. Conrad.
		\newblock The structure of solvable groups over general fields.
		\newblock Dans {\em Autour des sch\'{e}mas en groupes. {V}ol. {II}}, volume~46
		of {\em Panor. Synth\`eses}, pages 159--192. Soc. Math. France, Paris, 2015.
		
		\bibitem[GD67]{ega4}
		A. Grothendieck et J. Dieudonn{\'e}.
		\newblock {\em {\'E}l{\'e}ments de g{\'e}om{\'e}trie alg{\'e}brique {IV}},
		volumes 20, 24, 28, 32 de {\em Publications {M}ath{\'e}matiques}.
		\newblock Institut des {H}autes {\'E}tudes {S}cientifiques, 1964-1967.
		
		\bibitem[Gil04]{gille2004specialisation}
		Ph. Gille.
		\newblock Sp\'{e}cialisation de la {$R$}-\'{e}quivalence pour les groupes
		r\'{e}ductifs.
		\newblock {\em Trans. Amer. Math. Soc.}, 356(11)\string:4465--4474, 2004.
		
		\bibitem[KMT74]{kambayashi1974unipotent}
		T. Kambayashi, M. Miyanishi, et M. Takeuchi.
		\newblock {\em Unipotent algebraic groups}.
		\newblock Lecture Notes in Mathematics, Vol. 414. Springer-Verlag, Berlin-New
		York, 1974.
		
		\bibitem[Man68]{manin1968cubic1}
		Yu.~I. Manin.
		\newblock Kubicheskiye Giperpoverkhnosti. I. Kvazigruppy Klassov Tochek.
		\newblock {\em Izv. Akad. Nauk SSSR} Ser. Mat. 32 1968 1223--1244. (Russe)
		
		\bibitem[Man69]{manin1969cubic3}
		Yu.~I. Manin.
		\newblock Kubicheskiye Giperpoverkhnosti. III. Lupy Mufang i Ekvivalentnost' Brauera.
		\newblock {\em Mat. Sb. (N.S.)} 79 (121) 1969 155--170. (Russe)
		
		
		\bibitem[Man86]{manin1986cubic}
		Yu. I. Manin.
		\newblock {\em Cubic forms\string: algebra, geometry, arithmetic}, volume~4.
		\newblock Elsevier, 1986.
		
		\bibitem[Mer96]{merkurjev1996r-equivalence}
		A. Merkurjev. 
		\newblock $R$-equivalence and rationality problem for semisimple adjoint classical algebraic groups. 
		\newblock {\em Inst. Hautes \'Etudes Sci. Publ. Math.} No. 84 (1996), 189--213.
		
		\bibitem[Pla75]{platonov1975tannaka}
		V.~P. Platonov.
		\newblock On the {T}annaka-{A}rtin problem.
		\newblock {\em Dokl. Akad. Nauk SSSR}, 221(5)\string:1038--1041, 1975.
		
		\bibitem[Ros21]{rosengarten2021pathological}
		Z. Rosengarten.
		\newblock Pathological behavior of arithmetic invariants of unipotent groups.
		\newblock {\em Algebra Number Theory}, 15(7)\string:1593--1626, 2021.
		
		\bibitem[Ros57]{rosenlicht1957some}
		M. Rosenlicht.
		\newblock Some rationality questions on algebraic groups.
		\newblock {\em Ann.  Mat.  Pura  Appl.}, vol. 43, 25--50 (1957).
		
		\bibitem[Tit13]{tits2013resumes}
		J. Tits.
		\newblock {\em R\'{e}sum\'{e}s des cours au {C}oll\`ege de {F}rance
			1973--2000}, volume~12 des {\em Documents Math\'{e}matiques (Paris)}.
		\newblock Soci\'{e}t\'{e} Math\'{e}matique de France, Paris, 2013.
		
		\bibitem[Vos77]{voskresenskii1977reduced}
		V.~E. Voskresenski\u{\i}.
		\newblock The reduced {W}hitehead group of a simple algebra.
		\newblock {\em Uspekhi Mat. Nauk}, 32(6(198))\string:247--248, 1977.
		
		\bibitem[Vos77]{voskresenskii1977algebraic}
		V.~E. Voskresenski\u{\i}.
		\newblock {\em Algebraicheskiye tory}.
		\newblock  Nauka, Moscow, 1977. (Russe)
		
		\bibitem[Vos98]{voskresenskii2011algebraic}
		V.~E. Voskresenski\u{\i}.
		\newblock {\em Algebraic groups and their birational invariants}, volume 179 de
		{\em Translations of Mathematical Monographs}.
		\newblock American Mathematical Society, Providence, RI, 1998.
		\newblock Traduit du manuscrit russe par Boris \`E.
		Kunyavski\u{\i}.
		
	\end{thebibliography}
\end{document}